\documentclass[11pt,letterpaper]{article}

\usepackage[hscale=0.75,vscale=0.8]{geometry}

\usepackage{amsmath,amsthm,amssymb}
\usepackage{graphicx}
\usepackage{wrapfig}
\usepackage{thmtools, thm-restate}
\usepackage{bbm}
\usepackage{soul}
\newtheorem*{remark*}{Remark}

\newcommand{\blue}[1]{{#1}}

\usepackage[utf8]{inputenc} 
\usepackage[T1]{fontenc}    
\usepackage{hyperref}       
\usepackage{url}            
\usepackage{booktabs}       
\usepackage{amsfonts}       
\usepackage{nicefrac}       
\usepackage{microtype}      
\usepackage{xcolor}         
 
\usepackage[mathcal]{eucal}
\usepackage[round]{natbib}

\DeclareMathAlphabet\mathbfcal{OMS}{cmsy}{b}{n}
\newcommand{\BEAS}{\begin{eqnarray*}}
\newcommand{\EEAS}{\end{eqnarray*}}
\newcommand{\BEA}{\begin{eqnarray}}
\newcommand{\EEA}{\end{eqnarray}}
\newcommand{\BEQ}{\begin{equation}}
\newcommand{\EEQ}{\end{equation}}
\newcommand{\BIT}{\begin{itemize}}
\newcommand{\EIT}{\end{itemize}}
\newcommand{\BNUM}{\begin{enumerate}}
\newcommand{\ENUM}{\end{enumerate}}
\newcommand{\BA}{\begin{array}}
\newcommand{\EA}{\end{array}}

\newcommand{\R}{{\mathbb{R}}}

\newtheorem{theorem}{Theorem}
\newtheorem{lemma}[theorem]{Lemma}

\newtheorem{corollary}[theorem]{Corollary}

\newtheorem{remark}[theorem]{Remark}

\title{Mean first exit times of Ornstein--Uhlenbeck processes \\ in high-dimensional spaces}

\author{%
  Hans Kersting \\
  Inria, Ecole Normale Sup\'erieure \\
  PSL Research University, Paris, France. \\
  \texttt{hans.kersting@inria.fr}
  \and
  Antonio Orvieto \\
  Department of Computer Science\\
  ETH Zürich, Zürich, Switzerland.\\
  \texttt{antonio.orvieto@inf.ethz.ch} 
  \and
  Frank Proske \\
  Department of Mathematics \\
  University of Oslo, Oslo, Norway.\\
  \texttt{proske@math.uio.no}
  \and 
  Aurelien Lucchi \\
  Dpt. of Mathematics \& Computer Science \\
  University of Basel, Basel, Switzerland.\\
  \texttt{aurelien.lucchi@unibas.ch}
}

\begin{document}

\maketitle

\begin{abstract}
The $d$-dimensional Ornstein--Uhlenbeck process (OUP) describes the trajectory of a particle in a $d$-dimensional, spherically symmetric, quadratic potential. 
The OUP is composed of a drift term weighted by a constant $\theta \geq 0$ and a diffusion coefficient weighted by $\sigma > 0$. 
In the absence of drift (i.e. $\theta = 0$), the OUP simply becomes a standard Brownian motion (BM). 
This paper is concerned with estimating the mean first-exit time (MFET) of the OUP from a ball of finite radius $L$ for large $d \gg 0$.
We prove that, asymptotically for $d \to \infty$, the OUP takes (on average) no longer to exit than BM. 
In other words, the mean-reverting drift of the OUP (scaled by $\theta \geq 0$) has asymptotically no effect on its MFET. 
This finding might be surprising because, for small $d \in \mathbb{N}$, the OUP exit time is significantly larger than BM by a margin that depends on $\theta$. 
As it allows for the drift to be ignored, it might simplify the analysis of high-dimensional exit-time problems in numerous areas.

Finally, our short proof for the non-asymptotic MFET of OUP, using the Andronov--Vitt--Pontryagin formula, might be of independent interest.
%
\end{abstract}

\section{Introduction}

The $d$-dimensional Ornstein--Uhlenbeck process (OUP) is one of the most important stochastic processes.
It can be defined as a $d$-dimensional Brownian motion (aka.~Wiener process) where a drift is added which is proportional to the displacement from its mean.
It is usually defined as the solution of the following SDE:
\begin{equation} \label{eq:SDE_OUP}
    \mathrm{d}X_t = - \theta X_t \mathrm{d}t + \sigma \mathrm{d}B_t,
\end{equation}
where $\theta \in \mathbb{R}$ and $\sigma > 0$ are the drift and diffusion parameters and $B_t$ is a $d$-dimensional Brownian motion (BM).
Alternatively, if $X_0 = 0$, it can be defined as the unique $d$-dimensional zero-mean Gaussian process with the covariance function 
\begin{equation}
    \operatorname{cov}(X_{i,s},X_{j,t})
    =
    \delta_{ij} \cdot \frac{\sigma^2}{2 \theta} \cdot \left( e^{-\theta \vert t-s \vert} - e^{-\theta (t+s)} \right), 
\end{equation}
where $\delta_{ij}$ is the Kronecker delta and $X_{j,t}$ are the independent components of $X_t = (X_{1,t}, \dots, X_{d,t})$.

The OUP has been extensively studied \citep[Part II, Section 7]{borodin2002handbook}, sometimes with the parameter $\lambda = \frac{\theta}{\sigma^2} \in \mathbb{R}$ instead of $\theta$. 
Many authors define the OUP only for $\theta >0$, while others allow for $\theta \in \mathbb{R}$.
If $\theta \in \mathbb{R}$, one usually distinguishes between \blue{the} positively recurrent case ($\theta > 0$), the Brownian motion case ($\theta = 0$), and the transient case ($\theta < 0$); see p.~137 in \citet{borodin2002handbook}.
Here, we will allow for $\theta \in \mathbb{R}$ and only restrict this range if mathematically necessary. 

In this paper, we are interested in the mean \emph{first exit time} (aka.~first passage time) of a $d$-dimensional OUP from a ball of radius $L$, $B_L(0) = \{ x \in \mathbb{R}^d: \Vert x \Vert_2 < L \}$.
More specifically, we study the expectation $\mathbb{E}^x \tau_L$ of the stopping time $\tau_L = \inf_{t \geq 0} \{\Vert X_t \Vert_2 = L \}$, where the $x$ in $\mathbb{E}^x \tau_L$ indicates that $X_0 \in \R^d$ is chosen such that a.s. $\Vert X_0 \Vert = x$ for some $x \in [0,L]$. 
In particular, we want to understand the asymptotics of $\mathbb{E}^x \tau_L$ \emph{in high dimensions}, i.e.~as $d \to \infty$.

\blue{Such high-dimensional OUPs are important models in numerous applications ares.
For instance in biophysical sciences, high-dimensional OUPs are used to model the membrane potential of multiple neurons \citep{riccardi1979oup,faugeras2009oup}, the phylogenetic dynamics of quantitative traits \citep{butler2004phylogenetic,rohlfs2013modeling}, and single-cell differentiation for multiple cells \citep{matsumoto2016scoup}.
In complex systems, many neurons, traits, or cells have to be modeled~--~leading to high-dimensional OUP models, whose exit times mark when a specific average variation over all components is reached.}

\blue{
High-dimensional OUPs are also important in other scientific fields.
\citet[Section 3.4.]{grebenkov2014first} describes how MFETs matter in algorithmic trading, where the dimensionality depends on the number of different assets in a portfolio (which can get very large, e.g.~for exchange-traded funds). 
In economics, such models can be used to model the wealth of trading agents in an economy \citep{Ciolek2020highdimOUP}. }

\blue{
In non-convex optimization, high-dimensional Langevin equations are used as continuous-time models for stochastic gradient descent methods; these models are Ornstein--Uhlenbeck processes in local minima \citep{Li2017SMEs} whose dimensionality can go up to $d=10^7$ in modern machine-learning problems.
The first exit time of high-dimensional OUPs hence describe the time of escape from a spurious local minimum \citep{nguyen2019first,lucchi2022fOUP}.
}


MFETs of the OUP have of course already been studied to some extent in the literature. We refer the reader to the review paper by \citet{grebenkov2014first} for an overview and detailed literature review. Grebenkov proves a formula (see Eq.~(75) in his paper) for the MFET of a $d$-dimensional OUP for all $d \in \mathbb{N}$. 
Grebenkov's proof makes use of the Fokker-Planck equation and the eigenfunctions of the corresponding backward Laplace operator.
Previously, for the special case of $\sigma = 1$ and $\theta \geq 0$, \citet{Graczyk2008ExitTimesOUP} proved the same formula in their Theorem 2.2 by deriving and solving a suitable boundary value problem (BVP); see their Theorem 2.1.
In this paper, we will re-prove and extend the formulas by \citet{grebenkov2014first} and \citet{Graczyk2008ExitTimesOUP}.
\blue{We emphasize that our proof of Theorem 4 uses a strategy similar to \citet{Graczyk2008ExitTimesOUP}; we refer the reader to our discussion in Remark 5 for a comparison.}

While the previous formulas are \blue{satisfactory} for finite $d \in \mathbb{N}$, our reformulation in terms of the incomplete Gamma function enables us to get the asymptotic scaling for $d \to \infty$.
We will find that, as $d \to \infty$, the MFET of OUP is asymptotically equivalent to that of BM.
\begin{figure}[t]
    \centering
    \includegraphics[width=.23\columnwidth]{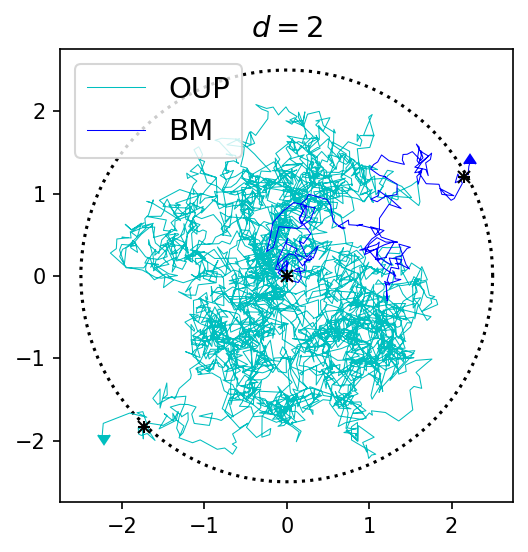}
    \includegraphics[width=.335\columnwidth]{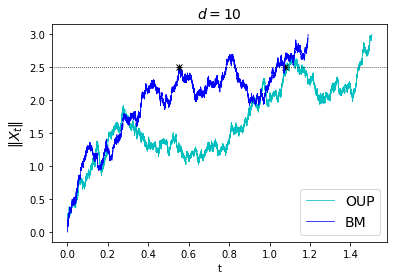}
    \includegraphics[width=.335\columnwidth]{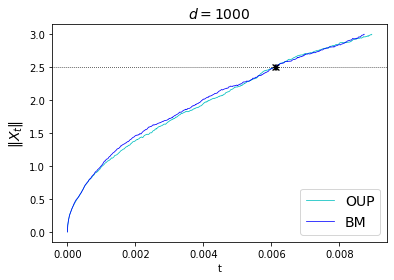}
    \caption{\emph{Motivation of this paper}: the scaling in $d$ of the first-exit time of OUP ($\theta > 0)$ and BM ($\theta = 0$) from a ball of radius $L$. In each of the three plots, a $d$-dim.~OUP and BM are run until they exit a ball of radius $L=2.5$ (visualized by the dotted black line). The point of the first exits is marked by black stars. The three plots are for $d=2$, $d=10$ and $d=1000$ respectively, with parameters $\sigma = 1.0$ and $\theta = 0.7$. On the left $(d=2)$, the entire trajectory is plotted in the 2-$d$ plane, while in the other plots only $\Vert X_t \Vert$ is plotted against time. We can see for $d=2$ how the trajectory of the OUP (i.e. the first-exit time) is much longer than of BM, due to the mean-reverting drift of the OUP. For $d=10$, this effect is already less pronounced, but still significant. For $d=1000$, the first exit times are almost the same. (For even larger $d$, the two lines becomes indistinguishable.) 
    This highlights what we will show in Corollary~\ref{cor:OUPlikeBM}, namely that (on average, for large $d$) OUP takes asymptotically no longer to exit than BM. 
    Note that we can already observe here that the exit times go to zero as $d$ grows, but it is a-priori not obvious how $\theta$ impacts the asymptotic rate (see Section~\ref{sec:discussion} for a discussion).
    Also, see Fig.~\ref{fig:roup_mfet_scaling_in_d}, where the mean asymptotics are depicted for the same parameters of $L$, $\sigma$, and $\theta$.
    }
    \label{fig:illustration_for_intro}
\end{figure}
The initial motivation for this paper is illustrated Fig.~\ref{fig:illustration_for_intro}: as $d$ increases the first exit times of both BM and OUP converge to zero at a rate that does not seem to differ between BM and OUP.
Our theroetical results and experiments (Fig.~\ref{fig:roup_mfet_scaling_in_d}) will confirm this.

\paragraph{Contributions} In this paper, we
\begin{enumerate}
    \item[(i)] reprove the known general formula for the MFET of a $d$-dimensional OUP (Theorem~\ref{thm:fet_roup}) by a shorter proof using the Andronov--Vitt--Pontryagin formula (see Remark~\ref{rmk:relation_to_prior_work} for the relation to prior work) and express this formula in terms of the incomplete Gamma function,
    \item[(ii)] prove asymptotically sharp bounds for the MFET (Theorem~\ref{th:exit_time_bounds_roup}), as $d \to \infty$, by relying on inequalities on the incomplete Gamma function from \citet{Neuman2013InequalitiesGammaFunction}, and
    \item[(iii)] demonstrate that (perhaps surprisingly) the MFET of OUP is asymptotically equal to the one of Brownian motion (Corollary~\ref{cor:OUPlikeBM}).
\end{enumerate}
\blue{Note that, while point (iii) might be initially surprising, it is much less surprising after careful examination; see our Discussion in Section~\ref{sec:discussion}.}

\subsection{Structure of paper}

The remainder of the paper is structured as follows.
First, in Section~\ref{sec:roup}, we explain how the MFET of the $d$-dimensional OUP can be considered as a MFET of the radial Ornstein--Uhlenbeck (rOUP) process.
Second, in Section~\ref{sec:avp_formula}, we introduce the Andronov--Vitt--Pontryagin (AVP) formula and its corresponding boundary value problem for the autonomous case.
Third, in Section~\ref{sec:mfet_oup}, we use both the rOUP and AVP formula to derive an explicit expression for the MFET of OUP and reformulate it in terms of the incomplete Gamma function.
Fourth, in Section~\ref{subsec:scaling_in_d}, we prove asymptotically sharp bounds for MFET (as $d \to \infty$) by exploiting inequalities from \citet{Neuman2013InequalitiesGammaFunction} on the incomplete Gamma function.
Fifth, in Corollary~\ref{cor:OUPlikeBM}, we show that these bounds imply that, as $d \to \infty$, the OUP does (on average) take no longer than Brownian motion to exit balls of arbitrary radius.
Sixth, in Section~\ref{sec:discussion}, we discuss the intuitive reasons and implications of our results.
Finally, we summarize and conclude in Section~\ref{sec:conclusion}.



\section{The radial Ornstein--Uhlenbeck process}
\label{sec:roup}

Let $X_t$ denote a $d$-dimensional OUP process, as defined in Eq.~\eqref{eq:SDE_OUP}.
We are concerned with the first exit time
\begin{equation}\label{eq:def_tau}
    \tau_L := \inf_{t \geq 0} \{\underbrace{\Vert X_t \Vert_2}_{=:\rho_t} = L\}.
\end{equation}
We can see that this stopping time only depends on $X_t$ via $\rho_t := \Vert X_t \Vert_2 = (\sum_{i=1}^d X_{i,t}^2)^{1/2}$.
The stochastic process $\rho_t$ is called the \emph{radial Ornstein--Uhlenbeck process (rOUP)}. It (as well as its square) can be represented by the following It\^o diffusions \citep[Appendices 1.25 and 1.26]{borodin2002handbook}.
\begin{lemma}
    Let $d \in \mathbb{N}$, $\theta \in \R$ and $\sigma > 0$.
    Consider the $d$-dimensional OUP $X_t$ from Eq.~\eqref{eq:SDE_OUP}. 
    Then, $\rho_t^2 = \sum_{i=1}^d X_{i,t}^2$ and $\rho_t = (\sum_{i=1}^d X_{i,t}^2)^{1/2}$ follow the SDEs
    \begin{align}
        \label{eq:sde_rho2}
        \mathrm{d} \rho_t^2  
        = 
        ( {\sigma^2 d} - 2\theta \rho_t^2)\, \mathrm{d}t + 2 \sigma \sqrt{\rho_t^2} \, \mathrm{d}B_t
    \intertext{and}
        \label{eq:sde_rho}
        \mathrm{d}\rho_t
        =
        \left[ \frac{(d-1)\sigma^2}{2\rho_t} - \theta \rho_t \right]\, \mathrm{d} t + \sigma \, \mathrm{d} B_t,
    \end{align}
    respectively.
    (Nb: Here $B_t$ is not the original BM, but another one (see proof). We however use the same symbol to declutter the notation.) 
\end{lemma}
\begin{proof}
    First, note that the components of $X_t = [X_{1,t}, \dots, X_{d,t}]$ are independent OUPs of the form 
    \begin{equation} \label{eq:sde_onedOUPs}
        \mathrm{d}X_{i,t}
        =
        -\theta X_{i,t}  +  \sigma \, \mathrm{d} B_{i,t}.
    \end{equation}
    Then, by $\rho_t^2 = \sum_{i=1}^d X_{i,t}^2$, application of It\^o's lemma yields
    \begin{align} \label{eq:sde_radialOU}
        \mathrm{d} \rho_t^2
        &=
        2 X_t^{\intercal} \, \mathrm{d}X_t
        +
        \frac12 \operatorname{tr} ( 2 I_d \overbrace{[dX_t dX_t^\intercal]}^{=\sigma^2 I_d \mathrm{d}t} )
        \\
        &=
        2 X_t^\intercal \mathrm{d}X_t + {\sigma^2 d} \, \mathrm{d}t.
        \label{eq:drt2}
    \end{align}
    The remaining term to compute is
    \begin{equation}
        X_t^\intercal \mathrm{d}X_t
        \, \overset{\eqref{eq:sde_onedOUPs}}{=} \,
        - \theta \rho_t^2 \, \mathrm{d}t
        +
        \sigma \sum_{i=1}^d X_{t,i} \, \mathrm{d}B_t^{(i)}
        \, = \,
        - \theta \rho_t^2 \, \mathrm{d}t
        +
        \sigma \sqrt{\rho_t^2} \, \mathrm{d}Y_t ,
        \label{eq:xttopdxt}
    \end{equation}
    where
    \begin{align}
        Y_t
        :=
        \sum_{i=1}^d
        \int_0^t
        \frac
        { X_{i,s} }
        {\sqrt{ \rho_t^2 }}
        \, \mathrm{d}B_s^{(i)}.
    \end{align}
    But $Y_t$ is a continuous martingale with quadratic variation $[Y,Y]_t = t$. This means, by L{\'e}vy's characterization of BM, that $Y_t$ is another Brownian motion. To declutter notation, we denote $Y_t$ by $B_t$, too.
    Now, insertion of Eq.~\eqref{eq:xttopdxt} into Eq.~\eqref{eq:drt2} yields Eq.~\eqref{eq:sde_rho2}.

    For the missing Eq.~\eqref{eq:sde_rho}, we observe that $\rho_t = \sqrt{\rho_t^2}$. Then, application of It\^o's lemma and insertion of Eq.~\eqref{eq:sde_rho2} gives Eq.~\eqref{eq:sde_rho}.
\end{proof}
Now, we have derived the SDE for the rOUP in Eq.~\eqref{eq:sde_rho}.
To exploit it, we will need to solve its Andronov--Vitt--Pontryagin formula which we will introduce next.

\section{The Andronov--Vitt--Pontryagin formula}
\label{sec:avp_formula}

The Andronov--Vitt--Pontryagin (AVP) formula is a formula to compute the expected exit times of an It\^o diffusion by use of its Kolmogorov backward operator.

Consider the It\^o SDE
\begin{equation} \label{eq:generic_sde}
    \mathrm{d} X_t = a(X_t,t) \, \mathrm{d}t + b(X_t,t) \, \mathrm{d}B_t,
\end{equation}
where $a: \R^d \times [0, T] \to \R^d$, $b: \R^d \times [0, T] \to \R^{d \times n}$, and $B_t$ is an $n$-dimensional Brownian motion.
The associated \emph{backward Kolmogorov operator} is defined as
\begin{equation}
L_x^\ast u(x,t) = \sum_{i=1}^n a_i(t) \frac{\partial u(x,t)}{\partial x^i} + \frac12 \sum_{i=1}^n \sum_{j=1}^n [b(X_t,t)b(X_t,t)^{\intercal}]_{ij} \frac{\partial^2 u(x,t)}{\partial x^i \partial x^j}.
\end{equation}

For any bounded domain $D \subset \R^d$, we define the first exit time from this domain:
\begin{equation} \label{def:rhoD}
    \pi_D := \inf \{ t > 0 \mid x(t) \not\in D \}.
\end{equation}
Note that, if $D$ is equal to a centered ball $B_L(0)$ of radius $L$, then $\pi_D = \tau_L$.
The next theorem characterizes the expected value of $\pi_D$.
\begin{theorem}[The Andronov--Vitt--Pontryagin formula] \label{thm:avp_formula}
    Let $D \subset \R^d$ be a bounded set, and assume that the following boundary value problem
    \begin{align} \label{avp:pde}
        \begin{split}
        \frac{\partial u(x,t)}{\partial t} + L^{\ast}_x u(x,t)
        &=
        -1, \quad \text{for } x \in D \text{ and } t \in \R, \\
        u(x,t)
        &=
        0, \text{ for } x \in \partial D,
        \end{split}
    \end{align}
    has a unique bounded solution $u(x,t)$.
    Then, the mean first passage time $\mathbb{E}^x \pi_D$ is finite and takes the form
    \begin{equation}
        \mathbb{E}^x \pi_D = u(x,0).
    \end{equation}
\end{theorem}
\begin{proof}
    See Theorem 4.4.3. in \citet{Schuss2010_StochasticProcesses}.
\end{proof}
Fortunately, in the autonomous case, the PDE \eqref{avp:pde} turns into an ODE which can be solved analytically in many cases.
\begin{corollary}[The Andronov--Vitt--Pontryagin boundary value problem for the autonomous case]
    \label{corollary:AVP_autonomous}
    Under the assumptions of Theorem \ref{thm:avp_formula}, if the coefficients $a$ and $b$ from Eq.~\eqref{eq:generic_sde} are independent of $t$, the solution $u$ of Eq.~\eqref{avp:pde} is also independent of $t$ and therefore solves the ODE
    \begin{align} \label{avp:ode}
        \begin{split}
            L^{\ast}_x u(x) &= - 1, \quad \text{for } x \in D. \\
            u(x) &= 0, \text{ for } x \in \partial D,
        \end{split}
    \end{align}
    In particular, $\mathbb{E}^x \pi_D = u(x)$.
\end{corollary}
\begin{proof}
    See Corollary 4.4.1. in \citet{Schuss2010_StochasticProcesses}.
\end{proof}
We can now apply the above Corollary to the SDE \eqref{eq:sde_rho} of the radial OUP. We will be able to solve the resulting ODE \eqref{avp:ode} and thereby obtain an exact formula for $\mathbb{E}^x \tau_L$.

\section{Mean first exit time of the $d$-dimensional OUP}
\label{sec:mfet_oup}

\begin{theorem} \label{thm:fet_roup}
    Let $L>0$ and let $x \in [0,L]$ such that a.s.~$\Vert X_0 \Vert = x$.
    Denote $\lambda := \frac{\theta}{\sigma^2}$.
    Then, for all $\sigma > 0$ and $\lambda \in \mathbb{R}$, we have that
    \begin{equation} \label{eq:fet_roup}
        \mathbb{E}^x \tau_L  =
        \frac{2}{\sigma^2} \int_x^L z^{1-d} e^{\lambda z^2}\left[ \int_0^z t^{d-1} e^{-\lambda t^2} \, \mathrm{d}t \right] \, \mathrm{d}z.
    \end{equation}
    If $\lambda > 0$, then this formula further simplifies to
    \begin{equation} \label{eq:fet_roup_with_gamma}
        \mathbb{E}^x \tau_L  =
        \frac{1}{\lambda^{\frac d2} \sigma^2} \int_x^L z^{1-d} e^{\lambda z^2} \gamma(d/2,\lambda z^2 ) \, \mathrm{d}z,
    \end{equation}
    where the function $\gamma$ denotes the upper incomplete Gamma function
    \begin{equation} \label{eq:def_gamma}
        \gamma(n,y)  :=
        \int_{0}^{y} t^{n-1} \exp(-t) \ \mathrm{d}t.
    \end{equation}
\end{theorem}

\begin{remark}[Relation to prior work]
    Theorem~\ref{thm:fet_roup} gives an exact representation of the desired MFET of the OUP for all $d \in \mathbb{N}$.
    Note that Theorem~\ref{thm:fet_roup} is already \emph{partially known}. 
    While both the physics and mathematics literature have independently derived variations of  Eq.~\eqref{eq:fet_roup}, our reformulation with the incomplete Gamma function \eqref{eq:fet_roup_with_gamma} is new.
    Importantly, this will enable us to derive novel lower and upper bounds on the MFET (see Theorem~\ref{th:exit_time_bounds_roup}).
    
    In physics, \citet{grebenkov2014first} derives our formula \eqref{eq:fet_roup} in his Eq.~(75), but in a different parametrization and with a different proof using the eigenfunctions of the backward Laplace operator.

    In mathematics, \citet{Graczyk2008ExitTimesOUP}, in their Theorem 2.2, also prove our formula \eqref{eq:fet_roup} with a similar strategy, but only for $\sigma = 1$ and $\theta \geq 0$ (called $\lambda$ in their notation). (Note that they use hypergeometric functions to express the integral in Eq.~\eqref{eq:fet_roup}.)
    Their proof strategy is similar to ours: our Corollary~\ref{corollary:AVP_autonomous} corresponds to their Theorem 2.1, which they derive by different means than \citet{Schuss2010_StochasticProcesses} used for our Corollary~\ref{corollary:AVP_autonomous}.
    Accordingly, our BVP~\eqref{eq:AVP_BVP_rho_t} appears on page 320 of \citet{Graczyk2008ExitTimesOUP}. 
    Hence, one can think of our Eq.~\eqref{eq:fet_roup} as an extension of Theorem 2.1 in \citet{Graczyk2008ExitTimesOUP} to all $\sigma > 0$ and all $\theta \in \R$, and of our proof as a shorter version of theirs.
    
    While our Eq.~\eqref{eq:fet_roup_with_gamma} is easily derived from Eq.~\eqref{eq:fet_roup} by a change of variable (see proof), it will be essential to prove the scaling for $d \to \infty$ below.
    Eq.~\eqref{eq:fet_roup_with_gamma} is thus an essential part of our analysis.
    \label{rmk:relation_to_prior_work}
\end{remark}

\begin{proof}[Proof of Theorem \ref{thm:fet_roup}]
    We first prove Eq.~\eqref{eq:fet_roup} and then Eq.~\eqref{eq:fet_roup_with_gamma}.
    The proof of Eq.~\eqref{eq:fet_roup} is split into three parts.
    First, we fix $x>0$ and show that the right-hand side of Eq.~\eqref{eq:fet_roup} solves the Andronov--Vitt--Pontryagin BVP, Eq.~\eqref{avp:ode}, associated with the radial Ornstein--Uhlenbeck process $\rho_t$, with $\rho_0 = x$.
    Second, we show that the solution to this BVP is unique.
    (Together, the first two steps imply that Eq.~\eqref{eq:fet_roup} holds for all $x>0$.)
    Third, we will show that it also holds for $x=0$ which will conclude the proof.

    We now go through the proof step by step.

    \emph{First step:} By the definition of the radial OU process $\rho_t = (\sum_{i=1}^d X_{i,t}^2)^{1/2}$, we have $\tau_L = \inf \{ t>0 : \rho_t = L \}$.
    Hence, by the SDE \eqref{eq:sde_rho} of $\rho_t$, the associated Andronov--Vitt--Pontryagin BVP \eqref{avp:ode} reads
    \begin{equation*}
        \left[ \frac{(d-1)\sigma^2}{2x} - \theta x \right] u^{\prime}(x) + \frac{\sigma^2}{2} u^{\prime \prime}(x)
        =
        -1,
        \qquad \text{ with } u(L) = 0,
    \end{equation*}
    or equivalently
    \begin{equation} \label{eq:AVP_BVP_rho_t}
        u^{\prime \prime}(x) 
        =
        \left[ 2 \lambda x - \frac{d-1}{x} \right] u^{\prime}(x) - \frac{2}{\sigma^2},
        \qquad \text{ with } u(L) = 0. 
    \end{equation}
    (Note that we here used that $\pi_{B_L(0)} = \tau_L$ by construction in Eq.~\eqref{def:rhoD}, so that we could apply the Andronov--Vitt--Pontryagin formula.)
    We will now show that the right-hand side of Eq.~\eqref{eq:fet_roup} solves the above equation, that is
    \begin{equation} \label{eq:u_eq_rhs}
        u(x)
        :=
        \frac{2}{\sigma^2} \int_x^L z^{1-d} e^{\lambda  z^2} \left[ \int_0^z t^{d-1} e^{-\lambda } \,\mathrm{d}t  \right] \,\mathrm{d}z.
    \end{equation}
    Its derivatives are
    \begin{align}
        u^{\prime}(x)
        &=
        -\frac{2}{\sigma^2} x^{1-d} e^{\lambda  x^2} I(x), \\
        u^{\prime \prime}(x)
        &= 
        -\frac{2}{\sigma^2} \left[ (1-d)x^{-d} e^{\lambda x^2} + 2 \lambda  x^{2-d} e^{\lambda x^2} \right] I(x)  -  \frac{2}{\sigma^2},
    \intertext{where}
        I(x) :&= \left[ \int_0^x t^{d-1}  e^{-\lambda  t^2} \,\mathrm{d}t \right].
    \end{align} 

    Insertion of these formulas yields Eq.~\eqref{eq:AVP_BVP_rho_t}:
    \begin{align} 
        \left[ 2 \lambda  x - \frac{d-1}{x} \right] u^{\prime}(x) - \frac{2}{\sigma^2}
        &=
        \left[ 2 \lambda  x - \frac{d-1}{x} \right]
        \left[-\frac{2}{\sigma^2} x^{1-d} e^{\lambda  x^2}\right] I(x) - \frac{2}{\sigma^2} \\
        &=
        - \frac{2}{\sigma^2} \left[ (1-d)x^{-d} e^{\lambda x^2} + 2 \lambda x^{2-d} e^{\lambda x^2} \right] I(x) - \frac{2}{\sigma^2}
        =
        u^{\prime \prime}(x).
    \end{align}

    \emph{Second step:} 
    To prove that $u$ from Eq.~\eqref{eq:u_eq_rhs} is the \emph{unique} solution to the second-order ODE \eqref{eq:AVP_BVP_rho_t} for any $x>0$, we recast it as a first-order ODE
    \begin{equation}
        \frac{ \mathrm d}{\mathrm d x} \begin{bmatrix} u_1(x) \\ u_2(x) \end{bmatrix}
        =
        \begin{bmatrix} u_2(x) \\  \left[ 2\lambda x - \frac{d-1}{x} \right] u_1(x) - \frac{2}{\sigma^2}  \end{bmatrix}
        =:
        f\left(x,\begin{bmatrix} u_1(x) \\ u_2(x) \end{bmatrix}\right).
    \end{equation}
    \blue{Now, on the domain $\tilde{x} \in [x,L]$, the vector field $f(\cdot,[u_1,u_2]^{\intercal})$, has Lipschitz uniformly bounded Lipschitz constants: 
    \begin{align}
    \begin{split}
        L(\tilde{x}) :&= \sup_{u \neq v \in \R^{2d}} \frac{\Vert f(\tilde{x},u) - f(\tilde{x},v) \Vert}{\Vert u - v \Vert} 
        = \min \left\{1, \vert 2 \lambda \tilde{x} - (d-1)/x \vert \right\} 
        \leq \left \vert 2 \lambda \tilde{x} - (d-1)/x \right \vert \\
        &\leq
        2 \lambda L + (d-1)/x < \infty, \qquad \forall \tilde{x} \in [x,L].
    \end{split}
    \end{align}
    Hence, the following global Lipschitz condition required for a global-Lipschitz version of Picard--Lindel\"of's theorem \citep[Corollary 2.6.]{Teschl2012} is satisfied for any choice of $x > 0$:
    \begin{equation}
        \int_x^L L(\tilde{x}) \, \mathrm{d}\tilde{x}  \leq
        (L-x) \cdot ( 2\lambda L + (d-1)/x ) < \infty
    \end{equation}
    Thus, this version of Picard--Lindel\"of's theorem yields that $u(x)$ is the unique solution of Eq.~\eqref{eq:AVP_BVP_rho_t}, for any fixed choice of $x>0$.
    }

    \emph{Third step:} To provide the missing case $x=0$, we first observe that $\lim_{l \to 0} \mathbb{E}^0 \tau_l = 0$ -- because otherwise the radial Ornstein--Uhlenbeck $\rho_t$ would stay at the origin for a positive time with positive probability, which it does not.
    Moreover, by the strong Markov property of $X_t$, we have for all $l \in (0,L]$ that
    \begin{equation}
        \mathbb{E}^0 \tau_L = \mathbb{E}^0 \tau_l + \underbrace{\mathbb{E}^l \tau_L}_{=u(l)}.
    \end{equation}
    Hence, in the limit $l \to 0$, we indeed obtain $\mathbb{E}^0 \tau_L = u(0)$ which concludes the proof of Eq.~\eqref{eq:fet_roup}.
    Note that all above steps worked for all $\lambda = \frac{\theta}{\sigma^2} \in \mathbb{R}$, and thus Eq.~\eqref{eq:fet_roup} holds for all $\lambda \in \mathbb{R}$.

    To prove the missing Eq.~\eqref{eq:fet_roup_with_gamma} for $\lambda > 0$, we observe by a change of variable, using the substitution function $\phi(t) := \sqrt{\frac{t}{\lambda}}$ with derivative $\phi^{\prime}(t) = \sqrt{4 \lambda t}$, that
    \begin{align} \label{eq:substitution_step}
        \begin{split}
        \int_0^{z = \phi( \lambda z^2)} t^{d-1} \exp(-\lambda t^2) \mathrm{d}t  &=
        \int_0^{\lambda z^2} \sqrt{4 \lambda t} \left( \frac{t}{\lambda} \right)^{\frac{d-1}{2}} \exp(-t) \mathrm{d}t \, \cdot \,  \\ &= 
        \frac 12 \lambda^{-d/2} \underbrace{\int_0^{\lambda z^2} t^{\frac d2 - 1} \exp(-t) \, \mathrm{d}t}_{=\gamma\left( \frac d2, \lambda z^2 \right)}.
        \end{split}
    \end{align}
    Insertion of the above formula into Eq.~\eqref{eq:fet_roup} yields Eq.~\eqref{eq:fet_roup_with_gamma}.
\end{proof}

\subsection{Scaling for $d \to \infty$}
\label{subsec:scaling_in_d}

Equipped with our formula ~\eqref{eq:fet_roup_with_gamma}, we will derive the scaling for dimension $d \to \infty$ by use of the following existing bounds on the incomplete Gamma function.
\begin{lemma}[Theorem 4.1.~by \citet{Neuman2013InequalitiesGammaFunction}]\label{theorem:Neuman2013}
    For the lower incomplete Gamma function $\gamma$ from Eq.~\eqref{eq:def_gamma}, the inequalities
    \begin{equation}
        \frac{x^a}{a} \exp\left(\frac{-ax}{a+1}\right)  \leq
        \gamma(a,x)  \leq
        \frac{x^a}{a(a+1)} (1+a\exp(-x))
    \end{equation}
    are valid for all $a>0$.
\end{lemma}
\begin{figure}
    \centering
    \includegraphics[width=.45\columnwidth]{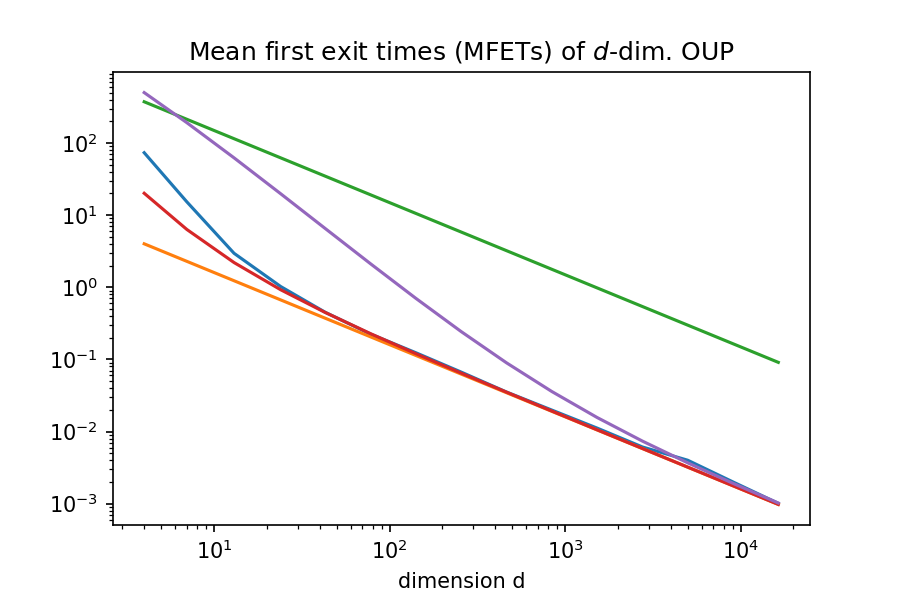}
    \includegraphics[width=.45\columnwidth]{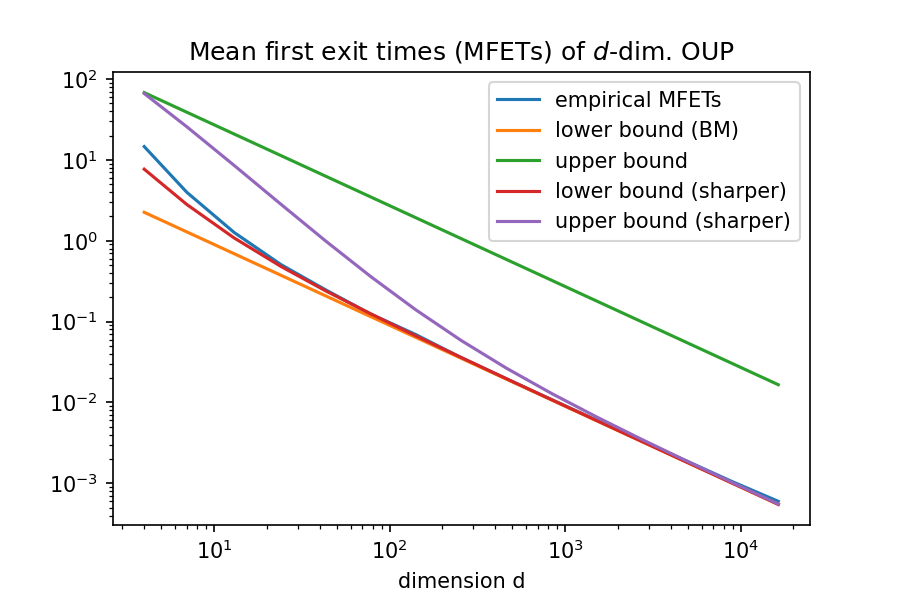}
    \caption{Log-log plot of the bounds of Theorem~\ref{th:exit_time_bounds_roup}. Parameters are set as $(L,x,\sigma,\lambda)=(4.0,0.0,1.0,0.5)$ (on the left plot) and as $(L,x,\sigma,\lambda)=(3.0,0.0,1.0,0.7)$ (on the right plot); but other parameters give the same behavior. The BM lower bound is the smaller lower bound that is independent of $\lambda$ (which is the exact MFET of BM). To compute the MFETs, 100 simulations with a step size of 0.001 (on the left) and of 0.0001 (on the right) were run for each choice of $d$. In both plots, the step size was chosen to reproduce the upper bounds (otherwise the point of exit can be missed by the discrete time steps); smaller choices also work. The bounds can be computed in closed form with our Eqs.~\eqref{bounds:tau_x_upper_boundI}--\eqref{bounds:eq_bla_II}. The bounds from Theorem~\ref{th:exit_time_bounds_roup} are verified by this plot. Remarkably, as we prove in Corollary~\ref{cor:OUPlikeBM}, the asymptotics of OUP $(\theta > 0)$ are equivalent to the BM lower bound ($\theta = 0$).} 
    \label{fig:roup_mfet_scaling_in_d}
\end{figure}
Now, we use the above result, to show the the scaling of $\mathbb{E}^x \tau_L$ in $d$.
The result of the following theorem is verified \blue{numerically} in Fig.~\ref{fig:roup_mfet_scaling_in_d}.
\begin{theorem} \label{th:exit_time_bounds_roup}
    Let $L>0$ and let $x \in [0,L]$ such that a.s.~$\Vert X_0 \Vert = x$.
    Denote $\lambda := \frac{\theta}{\sigma^2}$.
    Let $\theta > 0$ (or equivalently $\lambda = \frac{\theta}{\sigma^2} > 0$).
    Then, we have the following upper bounds
    \begin{align}\label{bounds:tau_x_upper_boundI}
        \mathbb{E}^x \tau_L \ &\leq \ 
        \frac{2}{\sigma^2 d (d+2)} \left( \frac{1}{\lambda} \left[  \exp(\lambda L^2) - \exp(\lambda x^2) \right] + \frac{d}{2} \left[ L^2 - x^2 \right] \right) \\ \ &\leq \
        \frac{1}{\theta}\left[ \exp(\lambda L^2) - \exp(\lambda x^2) \right] d^{-1},
        \label{bounds:tau_x_upper_boundII}
    \end{align}
    and lower bounds
    \begin{align}
        \label{bounds:eq_bla}
        \mathbb{E}^x \tau_L \ &\geq \
        \frac{1 + \frac 2d}{2\lambda \sigma^2} \left[ \exp\left(\frac{2 \lambda}{d+2} L^2\right) - \exp\left(\frac{2 \lambda}{d+2} x^2\right)  \right] \\ \ &\geq \
        \left[\frac{L^2 - x^2}{\sigma^2} \right] d^{-1}.
        \label{bounds:eq_bla_II}
    \end{align}
\end{theorem}
\begin{proof}
    We first show the upper bounds and then the lower bounds.

    \emph{Upper bounds:}
    By the second inequality from Lemma~\ref{theorem:Neuman2013}, we have 
    \begin{align} \label{upperboundproof:ineqI}
        \gamma(d/2,\lambda z^2)  \ &\leq \
        \frac{4 \lambda^{d/2} z^d}{d(d+2)} \cdot \left[ 1 + \frac{d}{2} \exp(-\lambda z^2) \right]
        \\ &\leq \
        \frac{4 \lambda^{d/2} z^d}{d(d+2)} \cdot \frac{2+d}{2} = \frac{2 \lambda^{d/2} z^d}{d}. 
        \label{upperboundproof:ineqII}
    \end{align}
    Insertion of the inequality \eqref{upperboundproof:ineqI} into Eq.~\eqref{eq:fet_roup_with_gamma} yields
    \begin{equation} \label{eq:first_proved_inequality}
        \mathbb{E}^x \tau_L  \ \leq \
        \frac{4}{\sigma^2 d (d+2)} \cdot \underbrace{\int_x^L z \exp\left(\lambda z^2\right)\left[ 1 + \exp(\lambda z^2) \right] \, \mathrm{d}z}_{= \int_x^L z \exp(\lambda z^2) \mathrm{d}z  + \frac d2 \int_x^L z \mathrm{d}z \ = \ \left[\frac{1}{2 \lambda} \exp(\lambda z^2) \right]_{z=x}^{z=L}+ \frac d4 [L^2 - x^2]},
    \end{equation}
    which is Eq.~\eqref{bounds:tau_x_upper_boundI}.
    For the missing upper bound \eqref{bounds:tau_x_upper_boundII}, we insert the less sharp inequality \eqref{upperboundproof:ineqII} into Eq.~\eqref{eq:fet_roup_with_gamma}:
    \begin{equation} \label{eq:second_proved_inequality}
        \mathbb{E}^x \tau_L  \ \leq \
        \frac{2}{\sigma^2 d} \int_x^L z \exp\left( \lambda z^2 \right) \, \mathrm{d}z =
        \frac{1}{\theta}\left[ \exp(\lambda L^2) - \exp(\lambda x^2) \right] d^{-1},
    \end{equation}
    where we used that $\lambda = \theta/\sigma^2$.
    Since the inserted upper bound from inequality \eqref{upperboundproof:ineqI} is sharper than the one from \eqref{upperboundproof:ineqII}, the first proved inequality \eqref{eq:first_proved_inequality} is lower than the second one \eqref{eq:second_proved_inequality}.

    \emph{Lower bounds:}
    By the first inequality from Lemma~\ref{theorem:Neuman2013}, we have
    \begin{equation} \label{bounds_for_eq_bla}
        \gamma(d/2,\lambda z^2)
        \ \geq \ 
        \frac{2}d (\lambda z^2)^{d/2} \exp\left(- \frac{d/2}{d/2 + 1} \lambda z^2\right)
        \ \geq \
        \frac{2}d (\lambda z^2)^{d/2} \exp\left(- \lambda z^2\right).
    \end{equation}
    Now, for the inequalities in Eqs.~\eqref{bounds:eq_bla} and \eqref{bounds:eq_bla_II}, we insert the first and second inequality from Eq.~\eqref{bounds_for_eq_bla} into Eq.~\eqref{eq:fet_roup_with_gamma}, respectively.
    For the first inequality, we thereby obtain
    \begin{align}
        \mathbb{E}^x \tau_L  \ &\geq \
        \frac{2}{\sigma^2 d} \int_x^L \exp\left( \frac{2}{d+2} \lambda z^2 \right) z \ \mathrm{d} z
        \ = \ 
        \frac{2}{\sigma^2 d} \left[ \frac{d+2}{4 \lambda} \exp\left(\frac{2 \lambda}{d+2} z^2\right) \right]_{z=x}^L \\
        &= \ 
        \frac{1 + \frac 2d}{2\lambda \sigma^2} \left[ \exp\left(\frac{2 \lambda}{d+2} L^2\right) - \exp\left(\frac{2 \lambda}{d+2} x^2\right)  \right].
        \label{eq:brzz1}
    \end{align}
    For the second inequality, we analogously obtain
    \begin{equation}
        \mathbb{E}^x \tau_L  \ \geq \
        \frac{2}{\sigma^2 d} \int_x^L z \ \mathrm{d}z =
        \left[\frac{L^2 - x^2}{\sigma^2}\right] d^{-1}.
        \label{eq:brzz2}
    \end{equation}
    As for the upper bounds, the ordering of inequalities in Eq.~\eqref{bounds_for_eq_bla} implies that the lower bound in Eq.~\eqref{eq:brzz2} is even lower than the one in Eq.~\eqref{eq:brzz1}.
\end{proof}
\begin{corollary}\label{cor:OUPlikeBM}
    Let $L>0$ and let $x \in [0,L]$ such that a.s.~$\Vert X_0 \Vert = x$.
    Then, we have, for $d \to \infty$, that
    \begin{equation}  \label{eq:corollary_corollary}
        \mathbb{E}^x \tau_L \ \sim \ \left[\frac{L^2 - x^2}{\sigma^2}\right] d^{-1}, \qquad \forall \, \theta \geq 0,
    \end{equation}
    i.e.~the MFET of a $d$-dimensional radial Ornstein Uhlenbeck process with \blue{arbitrary} $\theta > 0$ is asymptotically equivalent to the one of a $d$-dimensional Brownian motion with $\theta = 0$. 
    (See e.g.~Eq.~(7.4.2) in \citet{oksendal2003stochastic} for a proof that the MFET of a $d$-dimensional Brownian motion is equal to the right-hand side of Eq.~\eqref{eq:corollary_corollary}.)
\end{corollary}
\begin{proof}
    From Eq.~\eqref{bounds:tau_x_upper_boundI}, we know that
    \begin{equation}
        \mathbb{E}^x \tau_L \ \leq \ 
        \frac{2}{\lambda \sigma^2 d (d+2)} \left[  \exp(\lambda L^2) - \exp(\lambda x^2) \right] + \frac{1}{\sigma^2 (d+2)} \left[ L^2 - x^2 \right],
    \end{equation}
    where the first term is of order $O(d^{-2})$ and the second of order $O(d^{-1})$.
    Thus, the first term is irrelevant for the asymptotics and we have 
    \begin{equation}
        \limsup_{d \to \infty}
        \frac{\mathbb{E}^x \tau_L}{\sigma^2 d/(L^2 - x^2)}
        =
        \limsup_{d \to \infty}
        \frac{\mathbb{E}^x \tau_L}{\sigma^2 (d+2)/(L^2 - x^2)}
        \leq  1.
    \end{equation}
    On the other hand, by Eq.~\eqref{bounds:eq_bla_II}, we have that
    \begin{equation}
        \liminf_{d \to \infty}
        \frac{\mathbb{E}^x \tau_L}{\sigma^2 d/(L^2 - x^2)}
        \geq 1.
    \end{equation}
    This means that the liminf and the limsup coincide at $1$, i.e.~
    \begin{equation}
        \lim_{d \to \infty}
        \frac{\mathbb{E}^x \tau_L}{\sigma^2 d/(L^2 - x^2)}
        = 1,
    \end{equation}
    which is equivalent to the desired Eq.~\eqref{eq:corollary_corollary}.
\end{proof}

\section{Discussion}
\label{sec:discussion}

Theorem~\ref{th:exit_time_bounds_roup} provides asymptotically tight bounds for the MFET of a $d$-dimensional OUP from a ball of finite radius $L > 0$, as $d \to \infty$.
These bounds imply by virtue of Corollary~\ref{cor:OUPlikeBM} that, in high dimensions, the $d$-dimensional OUP takes (on average) no longer than a $d$-dimensional BM to exit a ball of radius $L$.
While we provided a rigorous proof above, this section gives \blue{some} intuition on whether this result is surprising and \blue{why the drift coefficient $\theta$ does not impact $\mathbb{E}^x \tau_L$ in high dimensions.}

\paragraph{Is Corollary~\ref{cor:OUPlikeBM} surprising?} \blue{(Short answer: At first, it might be; but a careful examination dispels the surprise.)} At first sight, Corollary~\ref{cor:OUPlikeBM} may be surprising.
After all, the mean-reverting drift of OUP (parametrized by $\theta > 0$) makes all the difference with the drift-less BM. 
In fact, due to the drift, the analytical properties of OUP and BM differ significantly; e.g., the OUP has a stationary distribution (unlike the BM) and its maximum grows in $\theta$ \citep{graversen2000maximal,jia2020moderate}.
But, our result suggests that~--~for the MFET~--~the drift becomes irrelevant as $d \to \infty$.
The following paragraphs contain an attempt to explain this.

First, we want to discern which part of our result might be surprising.
We proved that (i) the MFET $\mathbb{E}^x \tau_L$ goes to zero, and that (ii) $\mathbb{E}^x \tau_L$ is for any $\theta > 0$ asymptotically equivalent to $\mathbb{E}^x \tau_L$ when $\theta = 0$ (Brownian motion case).
We feel that, while point (i) is unsurprising, point (ii) may be surprising at first (before a closer examination on the drift below).

Regarding (i), it is clear that, as $d$ grows, the first-exit time $\tau_L$ will converge to zero (in \blue{probability}) for all choices of $\theta$ (or in fact, for any $d$-dimensional stochastic process with iid.~components). 
This can be seen from Eq.~\eqref{eq:def_tau}, where $\tau_L$ was defined as $\tau_L = \inf\{t>0: \Vert X_t \Vert^2_2 = \sum_{i=1}^d X_{i,t}^2 = L^2 \}$. 
Since the processes $X_{i,t}$ are iid., the sum $\sum_{i=1}^d X_{i,t}^2$ will likely reach $L^2$ faster for a high value of $d$. 
Thus, each value of the cumulative distribution function of $\tau_L$, as well as its mean, will monotonously decrease to zero, as $d \to \infty$.
But, from this, it does not follow how the asymptotic rates depend on drift coefficient $\theta$, as $d \to \infty$.
\blue{Next, we will explain why this is the case.}

\begin{figure}
    \centering
    \includegraphics[width=.45\columnwidth]{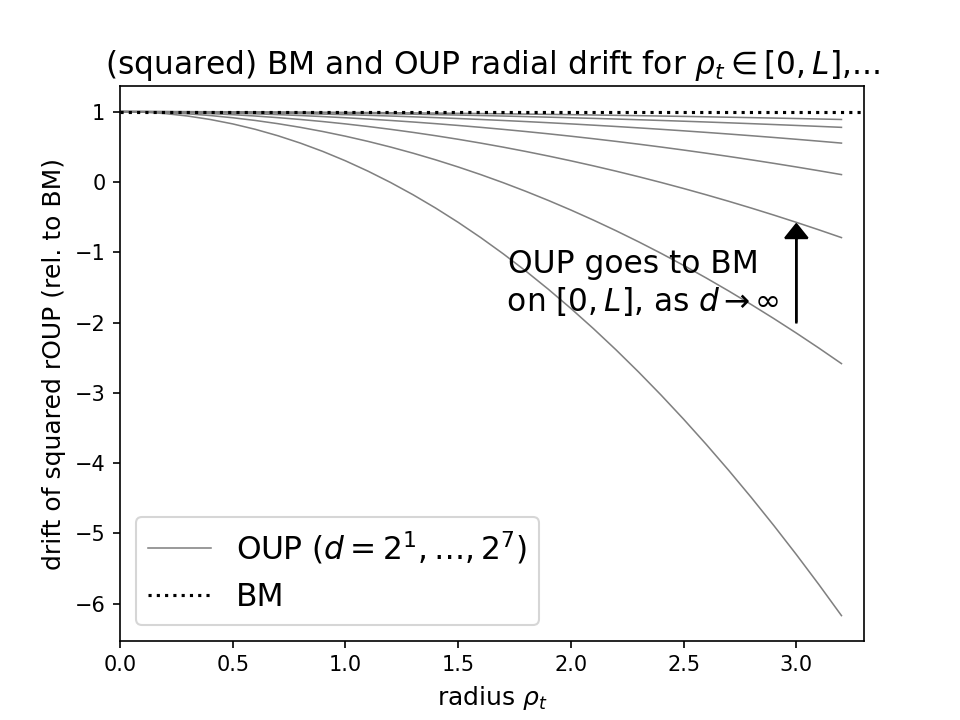}
    \includegraphics[width=.45\columnwidth]{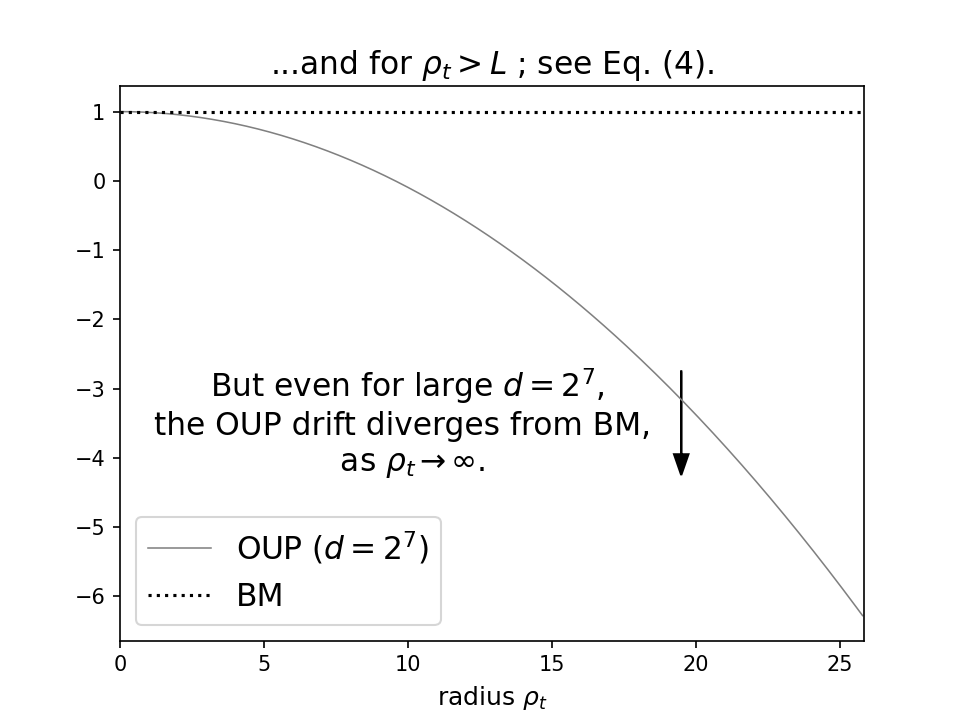}
    \caption{\blue{Simulations to give intuition for Corollary~\ref{cor:OUPlikeBM}. Parameters are set as $L=3.0$, $\sigma^2 = 1.0$, and $\theta = 0.7$. Both the \emph{left} and the \emph{right} plot show the drift of the squared radial OUP $\sigma^2 d - 2 \theta$ divided by the squared BM drift $\sigma^2 d$, i.e.~$\tfrac{\sigma^2 d - 2 \theta}{\sigma^2 d}$, on different domains; see Eq.~\eqref{eq:sde_rho2} for these drift terms. Hence, BM is a vertical line at $y=1$ and the OUPs have lower values (in some cases negative, i.e.~the drift pushes back to zero). 
    The \emph{left} plot shows these ratios for $d=2^1,2^2,...,2^7$ in gray, on $\rho_t \in [0,L]$. The $k$\textsuperscript{th} line from below belongs to $d=2^k$. We can see that, as $d \to \infty$, the OUP drift approaches BM. 
    However the \emph{right} plot shows that, for $\rho_t > L$, even the highest OUP drift $d=2^7$ still diverges to $-\infty$, as $\rho_t \to \infty$.
    In summary, the left plot explains why $\mathbb{E}^x \tau_L$ is asymptotically independent of $\theta$, while the right plot highlights that the rOUP still depends on $\theta$ for large values of $\rho_t$.
    More details in main text.}
    } 
    \label{fig:rebuttal_plots}
\end{figure}

\paragraph{\blue{For $\rho_t \in [0,L]$, the dynamics of the rOUP is asymptotically independent of $\theta$, as $d \to \infty$.}} 
\blue{From Eq.~\eqref{eq:def_tau}, we can rewrite the first exit time as $\tau_L = \sup_{t\geq 0}\{0 \leq \rho_s^2 \leq L^2, \forall s \in [0,t] \}$.
This means that, for $\tau_L$, only the dynamics of $\rho^2_t$ on the interval $\rho_t \in [0,L]$ matters.
But the SDE of $\rho_t^2$, Eq.~\eqref{eq:sde_rho2}, only depends on  $\theta$ through the SDE's drift coefficient $(\sigma^2 d - 2\theta\rho_t^2)$, and this drift coefficient becomes irrelevant for $d\to \infty$ on the bounded interval $\rho_t \in [0,L]$:
\begin{equation} \label{eq:shape_of_drift_coefficient}
    \sigma^2 d - 2\theta \rho_t^2
    \ \sim \
    \sigma^2 d, \qquad \text{ for } d \to \infty,
\end{equation}
because $\sigma^2$, $\theta$ and $L$ are chosen as constants independent of $d$.
Note that the right-hand side, $\sigma^2 d$, is indeed the drift of the $d$-dim.~squared Bessel process, i.e.~of the squared rOUP $\rho_t^2$ with $\theta = 0$ \citep[Eq.~(1.1)]{Pitman2018bessel}.
Hence, Eq.~\eqref{eq:shape_of_drift_coefficient} shows that the coefficients of the SDE of $\rho_t^2$, Eq.~\eqref{eq:sde_rho2}, is asymptotically independent of $\theta$ on $\rho_t \in [0,L]$.
The left subplot of Fig.~\ref{fig:rebuttal_plots} visualizes this effect. With this in mind, it is unsurprising that the MFET $\mathbb{E}^x \tau_L$ becomes independent of $\theta$, as $d\to\infty$.
}

\blue{
\paragraph{For $\rho_t \in \mathbb{R}_{>0}$, the dynamics of rOUP still depend on $\theta$.}
For unbounded $\rho_t$ the situation is however different.
If $\rho_t^2$ can take arbitrarily large values in the left-hand side of Eq.~\eqref{eq:shape_of_drift_coefficient}, then the asymptotics do not hold.
This is demonstrated on the right subplot of Fig.~\ref{fig:rebuttal_plots}.
In fact, as long as $L$ grows at least like $\mathcal{O}(\sqrt{d})$, we could still conclude from Eq.~\eqref{eq:shape_of_drift_coefficient} that the $d \to \infty$ asymptotics of $\mathbb{E}^x \tau_L$ will depend on $\theta$.}

\blue{
Therefore, it would be incorrect to conclude that Corollary~\ref{cor:OUPlikeBM}) implies that the $d$-dim.~squared rOUP $\rho_t^2$ will resemble the $d$-dim.~squared Bessel process on all of $\mathbb{R}_{>0}$, as $d \to \infty$.
But, on $[0,L]$, the SDE of the squared rOUP $\rho_t^2$ is asymptotically equivalent to the squared Bessel process.
(With this in mind, one can probably show more asymptotic similarities between these processes on compact domains.) 
}

\color{black}
\section{Conclusion}
\label{sec:conclusion}

In the above material, we proved two new theorems and one corollary.

First, Theorem~\ref{thm:fet_roup} gave two explicit formulas for the mean first exit time of a $d$-dimensional Ornstein--Uhlenbeck process from a ball of radius $L$.
The first of these formulas, Eq.~\eqref{eq:fet_roup}, coincides with the ones derived by prior work \blue{\citep{Graczyk2008ExitTimesOUP,grebenkov2014first}}, but our proof is very short (see discussion in Remark~\ref{rmk:relation_to_prior_work}) \blue{and leverages Andronov--Vitt--Pontryagin theory \citep{Schuss2010_StochasticProcesses}}.
The second of these formulas, Eq.~\eqref{eq:fet_roup_with_gamma}, is a novel reformulation in terms of the incomplete Gamma function.

Second, Theorem~\ref{th:exit_time_bounds_roup} exploits this reformulation by bounding the Gamma function as suggested by \citet{Neuman2013InequalitiesGammaFunction}.
The resulting bounds are new and are verified \blue{numerically} in Fig.~\ref{fig:roup_mfet_scaling_in_d}. 
Since the upper and lower bounds are asymptotically equivalent, they have the (perhaps \blue{at first} surprising) implication (Corollary~\ref{cor:OUPlikeBM}) that, for $d \to \infty$, the $d$-dimensional OUP takes no longer than a $d$-dimensional Brownian motion to exit a ball of arbitrary radius $L$.
Thus, for large $d$, the drift does not matter for the (mean) exit time of OUP.
This might be surprising because the $d$-dimensional Brownian motion is just an OUP without drift, i.e.~with $\theta = 0$, and one would expect that a larger drift back to zero leads to a slower exit (as is the case for small $d \in \mathbb{N}$).
The simulations in Fig.~\ref{fig:roup_mfet_scaling_in_d} verify this asymptotic relation between OUP and Brownian motion.
\blue{In Section~\ref{sec:discussion}, we then give intuition to disperse any initial surprise that readers might have experienced.}

Our findings shed light on some unusual behavior of the OUP in high dimensions.
We hope that our Corollary~\ref{cor:OUPlikeBM} will be applicable in the numerous research areas (biology, economics, machine learning, statistical mechanics, etc.), where high-dimensional OUPs are used for modeling. 

\subsection*{Acknowledgements}

We thank Ralf Metzler for his helpful comments on an early version of the manuscript. We thank the anonymous reviewers for their careful remarks which helped us to improve this paper.

{\small
\bibliographystyle{plainnat}
\bibliography{bibfile}
}

\end{document}